\documentclass[10pt]{article}
\usepackage{mathrsfs}
\usepackage{amsfonts}
\usepackage{amsthm,amsmath,amssymb,anysize}
\newtheorem{lemma}{Lemma}[section]
\newtheorem{theorem}[lemma]{Theorem}

\newtheorem{proposition}[lemma]{Proposition}
\newtheorem{definition}[lemma]{Definition}
\newtheorem{corollary}[lemma]{Corollary}
\usepackage[numbers,sort&compress]{natbib}

\marginsize{3.6cm}{3.6cm}{3.6cm}{3cm}
\numberwithin{equation}{section}

\title{\textsf{Derivations of the Positive Part of the Two-parameter Quantum Group of type $G_2$}}
\author{
\textsc{Yongyue Zhong and  Xiaomin Tang }   \\
      }
\date{ }
\begin{document}
\maketitle
\begin{quotation}
\noindent\textbf{Abstract.}
  In this paper, we compute the derivations of the positive part of the two-parameter quantum group of type $G_2$ by embedding it into a quantum torus. We also show that the Hochschild cohomology group of degree $1$ of this algebra is a two dimensional vector space over the complex field.
\vskip 5pt

\end{quotation}

\footnotetext{


 This work was supported in part by the NNSF of China [grant number 11771069],  and the Fund of Heilongjiang Provincial Laboratory of the Theory and Computation of Complex Systems.
}

  \setcounter{section}{0}
\section{Introduction}
 \quad
 Throughout the paper, the symbols $\mathbb{C}$,$\mathbb{N}$ and $\mathbb{Z}$ represent for the set of complex numbers, nonnegative integers and integers, respectively.  Let $\mathfrak{g}$ be a finite dimensional complex simple Lie algebra and let $r, s \in \mathbb{C}^*:=\mathbb{C}\setminus\{0\}$. The two-parameter quantum groups (or quantized enveloping algebras) $U_{r,s}(g)$  have been introduced in the literatures \cite{Bergeron, Benk1,Benk2} and the references therein. Overall, the two-parameter quantized enveloping algebras $U_{r,s}(g)$ are close analogues of their one-parameter peers, and share a similar structural and representation theory as the one-parameter quantized enveloping algebras $U_q(g)$.
For instance, the author in \cite{tangxin0} prove that the positive part of the two-parameter quantized enveloping algebra $U_{r,s}^+(g)$ is isomorphic to certain two-parameter Ringel-Hall algebra.

 Nonetheless, there are some differences between the one-parameter quantum groups and two-parameter quantum groups. The two-parameter quantum groups $U_{r,s}(g)$ are more rigid in that they possess less symmetry. In particular, the center of two-parameter quantum groups $U_{r,s}(g)$ already posed a different picture \cite{Benk3}.  On the other hand, these differences also make it plausible to more effectively study the structures of these algebras.
   In \cite{Hu2007}, the author has studied the convex PBW-type bases of two-parameter quantum groups of typical Lie algebra $B_n$, $F_4$ and $G_2$, respectively. It frequently inspires new directions of study in the two-parameter quantum groups. Essentially, the two-parameter quantized enveloping algebra $U_{r, s}(\mathfrak{g})$ and one-parameter quantized enveloping $U_q(\mathfrak{g})$ have similar structural and representation theory. In \cite{Fan}, Fan studied the intimate relationship revealed between a two-parameter quantum algebra and its one-parameter analog by the specialization at $t=1$.

In recent years, more studies have been conducted toward their subalgebras such as $U_{r,s}^+(g)$, and the augmented Hopf algebras $U_{r,s}^{\ge 0}(g)$. The authors has computed the derivations of the two-parameter quantized enveloping algebra $U_{r,s}^+(B_2)$ and the algebra automorphisms of the Hopf algebra $\breve{U}_{r,s}^{\geq 0}(sl_3)$ in \cite{tangxin1,tangxin2}. A similar work has been carried out for the algebra $U^+_{r,s} (B_3)$ and the Hopf algebra $\breve{U}^{\geq 0}_{r,s} (B_3)$ in \cite{limin}. The same problem on subalgebra of $U_{r,s}(G_2)$ is worth studying, i.e., to compute the derivations of  $U_{r,s}^+(G_2)$ and the algebra automorphisms of  $\breve{U}_{r,s}^{\geq 0}(G_2)$. In particular, in \cite{tangxin3} the author study the group of algebra automorphisms  for $\breve{U}_{r,s}^{\geq 0}(G_2)$ and prove that the group of Hopf algebra automorphisms  for $\breve{U}_{r,s}^{\geq 0}(G_2)$ is isomorphic to a torus of rank two. But the structure of derivations for  $U_{r,s}^+(G_2)$ was still unknown. In this paper, we will study this problem, that is, we want to determine the structure of derivation algebra of $U_{r, s}^+(G_2)$. In order to calculate the derivation, we first embed $U_{r, s}^+(G_2)$ into a quantum torus with know derivations, and then via this embedding we shall be able to pull the information on derivations back to the algebra $U_{r, s}^+(G_2)$.

\section{Some basic properties of the $U_{r,s}^+ (G_2)$}

In this section, we recall the definitions of the algebra $U_{r,s}^+ (G_2)$.
We also present some facts and notation which will be used throughout the paper.

\subsection{The algebra $U_{r, s}^+(G_2)$}\label{section2.1}

\qquad In the paper, we will always assume that the parameter $r, s$ are chosen from $\mathbb{C}^*$ such that $r^ms^n=1$ implies $m=n=0$ and use the notations $\xi:=r^2-s^2+rs$, $\eta:=r^2-s^2-rs$, $\zeta:=(r^3-s^3)(r+s)^{-1}$, $[\![k, l]\!]:=\{k, k+1, ....,l-1, l\}$ for any positive integers $k$ and $l$ with $k<l$. 
\quad  Recall that two-parameter quantum group $U_{r, s}(G_2)$ associated to the complex simple Lie algebras of type $G_2$ have studied in \cite{Hu2007}.

\begin{definition}
The two-parameter quantized  group $U^+:=U^+_{r, s}(G_2)$ is defined to be the $\mathbb{C}$-algebra generated by the generators  generators $e_1$ and $e_2$,
subject to the quantum Serre relations:
\begin{eqnarray*}
&&e_2^2e_1-(r^{-3}+s^{-3})e_2e_1e_2+(rs)^{-3}e_1e_2^2=0, \label{r1}\\
&&e_1^4e_2-(r+s)(r^2+s^2)e_1^3e_2e_1+(rs)^6e_2e_1^4\\
 &&\ \ \ +rs(r^2+s^2)(r^2+rs+s^2)e_1^2e_2e_1^2 -(rs)^3(r+s)(r^2+s^2)e_1e_2e_1^3=0.\label{r2}
\end{eqnarray*}
\end{definition}

The following definition of a good numbering is given by Lusztig in [16, Subsection 4.3]:
\begin{definition}
Let $\Phi$ be a root system of a complex simple Lie algebra with $\Phi^+$ a positive root system and $\Pi=\{\alpha_1,...,\alpha_n\}$ a base of simple roots and $i\in [\![1, n]\!]$. A simple root $\alpha_i \in \Pi$ is said to be special if the coefficient with which $\alpha_i$ appears in any $\beta \in \Phi^+$ (expressed as $\mathbb{N}$-linear combination of simple roots) is $\leq$ 1. A simple root $\alpha_i \in \Pi$ is said to be semispecial if the coefficient with which $\alpha_i$ appears in any $\beta \in \Phi^+$
is $\leq$ 1, except for a single root $\beta$, for which the coefficient is necessarily $2$.
The numbering of the rows and columns of the Cartan matrix (or, equivalently, of $\Pi$) has been, so far, arbitrary. We say that this numbering is
good if for any $i\in [1, n]$, $\alpha_i$ is special or semispecial when considered as a simple root of the root system $\Phi\cap(\mathbb{Z}\alpha_1+ \cdots + \mathbb{Z}\alpha_i)$.
\end{definition}

We can always choose a good numbering for the rows and columns of the Cartan matrix. The following lemma tell us
a good numbering of $U^+$.
\begin{lemma}\label{tangz01}
( see \rm{\cite{Hu2009}})
Let $\Phi^+$ be a positive root system of the complex simple Lie algebra of type $G_2$. Then
$\Phi^{+}=\{\alpha_1, \alpha_2, \alpha_1+\alpha_2, \alpha_2+2\alpha_1, \alpha_2+3\alpha_1, 2\alpha_2+3\alpha_1\}$
is a convex ordering of positive roots, $\Pi=\{\alpha_2, \alpha_1\}$ is a good numbering.
\end{lemma}

\quad With the notation as above, we take $X_1, \cdots, X_6$ as the corresponding quantum root vectors in $U^+$, where
\begin{eqnarray*}
&& X_1 = e_1,\ \ \ X_2=e_1X_3-r^2sX_3e_1,\ \ \ X_3=e_1X_5-rs^2X_5e_1,\\¡¡
&& X_4=X_3X_5-r^2sX_5X_3,\ \ \ X_5=e_1e_2-s^3e_2e_1,\ \ \ X_6= e_2.
\end{eqnarray*}
Observe that $\{X_6^{n_6}X_5^{n_5}X_4^{n_4}X_3^{n_3}X_2^{n_2}X_1^{n_1}|n_i \in \mathbb{N}, i=1, \cdots, 6\}$ form a convex PBW-type Lyndon basis of the algebra $U^+$ (see \cite{Hu2010}, Theorem 2.4).

We can get the following formula by simple calculation.
\begin{lemma}\label{12}
The following identities hold:

 $X_6X_5=r^{-3}X_5X_6$,

 $X_6X_4=r^{-6}s^{-3}X_4X_6-r^{-3}s^{-3}(r-s)(r^2-s^2)X_5^3$,

 $X_6X_3=r^{-3}s^{-3}X_3X_6-r^{-2}s^{-3}(r^2-s^2)X_5^2$,

 $X_6X_2=r^{-3}s^{-6}X_2X_6-r^{-1}s^{-5}(r^3-s^3)X_5X_3-r^{-2}s^{-6}\eta X_4$,

 $X_6X_1=s^{-3}X_1X_6-s^{-3}X_5$,

 $X_5X_4=r^{-3}X_4X_5$,

 $X_5X_3=r^{-2}s^{-1}X_3X_5-r^{-2}s^{-1}X_4$,

 $X_5X_2=r^{-3}s^{-3}X_2X_5-r^{-2}s^{-3}\zeta X_3^{2}$,

 $X_5X_1=r^{-1}s^{-2}X_1X_5-r^{-1}s^{-2}X_3$,

 $X_4X_3=r^{-3}X_3X_4$,

 $X_4X_2=r^{-6}s^{-3}X_2X_4-r^{-3}s^{-3}(r-s)\zeta X_3^3$,

 $X_4X_1=r^{-3}s^{-3}X_1X_4-r^{-2}s^{-3}\zeta X_3^2$,

 $X_3X_2=r^{-3}X_2X_3$,

 $X_3X_1=r^{-2}s^{-1}X_1X_3-r^{-2}s^{-1}X_2$,

 $X_2X_1=r^{-3}X_1X_2$.
\end{lemma}

\subsection{$U^+$ as an Iterated Ore Extension}

The aim of this section is to find the algebraic structures of the algebras $U^+$ as an Ore extension. Firstly, let us recall the definition of Ore extension. Let $R$ be an algebra and $R[t]$ be the free (left) $R$-module consisting of all polynomials of form $$P=a_nt^n+a_{n-1}t^{n-1}+...+a_0t^{0}$$
with coefficients in $R$. If $a_n \neq 0$, we say that the degree ${\rm \deg}(P)$ of $P$ is equal to $n$; by convention, we set ${\rm \deg}(0)=-\infty$. Let $\sigma$ be an algebra endomorphism of $R$. An $\sigma$-derivation of $R$ is a linear endomorphism $\delta$ of $R$ such that
\begin{equation*}
\delta(ab)=\sigma(a)\delta(b)+\delta(a)b
\end{equation*}
for all $a, b \in R$.
Let $R$ be an algebra without zero-divisors. Given an injective algebra endomorphism $\sigma$ of $R$ and an $\sigma$-derivation $\delta$ of $R$, there exists a unique algebra structure on $R[t]$ such that the inclusion of $R$ into $R[t]$ is an algebra, morphism and relation $ta=\sigma(a)t+\delta(a)$ holds for all $a$ in $R$, the algebra defined by above, denote $R[t; \sigma,\delta]$, is called the {\it Ore extension}, also called {\it a  skew  polynomial algebra} \cite{Christian}.

Using Lemmas \ref{tangz01} and \ref{12}, $U^+$  can be regarded as an iterated Ore extension as follows.

\begin{lemma}\label{tangz02}
$U^+$ is a skew polynomial ring which could be expressed as:
\begin{equation*}
U^+=\mathbb{C}[X_1][X_2; \sigma_2, \delta_2][X_3; \sigma_3, \delta_3]...[X_6; \sigma_6, \delta_6]
\end{equation*}
where the $\sigma_j$'s are $\mathbb{C}$-linear automorphisms and the $\delta_j$'s are $\mathbb{C}$-linear $\sigma_j$-derivations such that for $1\leq i< j \leq 6$,$\sigma_j(X_i)=\lambda_{j,i}X_i$ and $\delta_j(X_i)=P_{j,i}$, with $\lambda_{j,i}$ and $P_{j,i}$ defined as follows:
\begin{equation*}\label{lambda}
X_jX_i-\lambda_{j,i}X_iX_j=P_{j,i}.
\end{equation*}
\end{lemma}
Observe that we can obtain $\lambda_{j,i}$ and $P_{j, i}$ for $1\leq i< j \leq 6$ from Lemma \ref{12} and thus determine these automorphism $\sigma_j$ and $\sigma_j$-derivation $\delta_j$. In fact, $P_{j,i}$ are of the form
\begin{equation*}\label{P}
P_{j,i}=\Sigma_{\overline{k}=(k_{i+1},...,k_{j-1})}c_{\overline{k}}X_{i+1}^{k_{i+1}}...X_{j-1}^{k_{j-1}}
\end{equation*}
where $c_{\overline{k}} \in \mathbb{C}$.

\begin{corollary}\label{3.1}
For $j\in [\![2, 6]\!]$, the $\mathbb{C}$-linear automorphisms $\sigma_j$ and $\sigma_j$-derivation $\delta_j$ which appear in Lemma \ref{tangz02}, have the following commutative relations:
\begin{eqnarray*}
\sigma_6 \circ \delta_6 =r^{-3}s^3\delta_6 \circ \sigma_6,  \ \ \ \ \
\sigma_5 \circ \delta_5 =r^{-1}s\delta_5 \circ \sigma_5, \\
\sigma_4 \circ \delta_4 =r^{-3}s^3\delta_4 \circ \sigma_4, \ \ \ \ \
\sigma_3 \circ \delta_3 =r^{-1}s\delta_3 \circ \sigma_3.
\end{eqnarray*}
Furthermore, for any $\ell \in [\![2, 6]\!]$, there exists some $q_\ell \in \mathbb{C}^*$ such that $\sigma_\ell \circ \delta_\ell=q_\ell\delta_\ell \circ \sigma_\ell$, in which $q_i=r^{-3}s^{3}$ for $i \in \{4, 6\}$ and $q_j=r^{-1}s$ for $j \in \{3, 5\}$.
\end{corollary}
\begin{proof}
According to Lemma \ref{12}, we can list the values of $\lambda_{j,i},P_{j,i}$ ($1\le i< j\le 6$) in Lemma \ref{tangz02} as
\begin{eqnarray*}
&& \lambda_{2,1}=r^{-3},\ \lambda_{3,1}=r^{-2}s^{-1},\ \lambda_{3,2}=  r^{-3},\ \lambda_{4,1}=r^{-3}s^{-3}, \\
&& \lambda_{4,2}= r^{-6}s^{-3},\  \lambda_{4,3}=r^{-3},\ \lambda_{5,1}=r^{-1}s^{-2},\ \lambda_{5,2}=r^{-3}s^{-3},\\
&& \lambda_{5,3}=r^{-2}s^{-1},\ \lambda_{5,4}= r^{-3},\ \lambda_{6,1}=s^{-3},\ \lambda_{6,2}= r^{-3}s^{-6}, \\
&& \lambda_{6,3}=r^{-3}s^{-3},\ \lambda_{6,4}=r^{-6}s^{-3},\ \lambda_{6,5}=r^{-3}
\end{eqnarray*}
and
\begin{eqnarray*}
&& P_{2,1}=0,\ P_{3,1}=-r^{-2}s^{-1}X_2, \  P_{3,2}=0, P_{4,1}=-r^{-2}s^{-3}\zeta X_3^2, \\
&&  P_{4,2}=-r^{-3}s^{-3}\zeta(r-s)X_3^3, \ P_{4,3}=0, \ P_{5,1}=-r^{-1}s^{-2}X_3\zeta X_3^2, \\
&& P_{5,2}=-r^{-2}s^{-3}\zeta X_3^2, \ P_{5,3}=-r^{-2}s^{-1}X_4, \ P_{5,4}=0, \\
&&  P_{6,1}=-s^{-3}X_5,\ P_{6,2}=-r^{-1}s^{-5}(r^3-s^3)X_5X_3-r^{-2}s^{-6}\eta X_4,\\
&& P_{6,3}=-r^{-2}s^{-3}(r^2-s^2)X_5^2, \ P_{6,4}=-r^{-3}s^{-3}(r-s)(r^2-s^2)X_5^3, \ P_{6,5}=0.
\end{eqnarray*}
By the above equations and (\ref{lambda}), we know $\sigma_6 \circ \delta_6 (X_5)=\delta_6 \circ \sigma_6 (X_5)=0$; in additions,
\begin{equation*}
 \sigma_6 \circ \delta_6 (X_4)=\sigma_6(-r^{-3}s^{-3}(r-s)(r^2-s^2)X_5^3)=-r^{-3}s^{-3}(r-s)(r^2-s^2)r^{-9}X_5^3
\end{equation*}
and
\begin{equation*}
\delta_6 \circ \sigma_6 (X_4)=\delta_6(r^{-6}s^{-3}X_4)=r^{-6}s^{-3}(-r^{-3}s^{-3}(r-s)(r^2-s^2)X_5^3)
\end{equation*}
yield that $\sigma_6 \circ \delta_6 (X_4)=r^{-3}s^3\delta_6 \circ \sigma_6 (X_4)$. Similarly,
let $\sigma_6 \circ \delta_6$ and $\delta_6 \circ \sigma_6$ action on $X_3$, $X_2$ and $X_1$ respectively,  we have $\sigma_6 \circ \delta_6 (X_i)=r^{-3}s^3\delta_6 \circ \sigma_6 (X_i), i=1,2,3$. This shows that $\sigma_6 \circ \delta_6 =r^{-3}s^3\delta_6 \circ \sigma_6$.  The proof for other cases is similar. The proof is completed.
\end{proof}

\subsection{Embedding $U^+$ into a Quantum Torus}\label{section2.3}

\ \ \ For our proof, $U^+$ should be embedded in a quantum group of known derivations. We construct the so-called ``Derivative-Elimination Algorithm'', it consists of a sequence of changes of variables inside the division ring $F = Fract(R)$, starting with the indeterminates ($X_1$ ,...,$X_N$ ) and terminating with new variables ($T_1$ ,...,$T_N$ ) \cite{cauchon}. We fix some notation. For any $n \in \mathbb{N}^*$, $q \in \mathbb{C}$, we set $[n]_q:=1+q+\cdot \cdot \cdot+q^{n-1}$ and $[n]!_q:=[1]_q\times \cdot \cdot \cdot \times [n]_q$. In addition , we set $[0]!_q:=1$.

Recall that $U^+=\mathbb{C}[X_1][X_2;\sigma_2, \delta_2][X_3; \sigma_3, \delta_3]\cdot \cdot \cdot [X_6;\sigma_6, \delta_6]$ is an iterated Ore extension.
The point is that the algebra $U^+_{r, s} (G_2)$ has a Goldie quotient ring, which we shall denote by $Frac(U^+)$. Within the Goldie quotient ring $Frac(U^+)$ of $U^+$, for all $i\in [\![1, 6]\!]$, we set $X_i^{(7)}:=X_i$, $\sigma_i^{(7)}:=\sigma_i$, $\delta_i^{(7)}:=\delta_i$ and define $X_i^{(\ell)}\in Frac(U^+)$, for $\ell \in [\![2, 6]\!]$, by the following formulas ( see Subsections 3.2 in \cite{cauchon}):
\begin{equation}\label{tangz03}
X_i^{(\ell)}=\begin{cases}
X_i^{(\ell+1)}, & i\geq l;\\
\sum_{n=0}^{+\infty}\frac{(1-q_\ell)^{-n}}{[n]!_{q_\ell}}((\delta_\ell^{(\ell+1)})^n\circ (\sigma_\ell^{(\ell+1)})^{-n}(X_i^{(\ell+1)}))(X_\ell^{(\ell+1)})^{-n}, & i <l,
\end{cases}
\end{equation}
where $q_{\ell}$ are the same as in Corollary \ref{3.1} and for $j \in [\![1, 6]\!]$, the $\sigma_j^{(\ell+1)}$'s are $\mathbb{C}$-linear automorphisms and the $\delta_j^{(\ell+1)}$'s are $\mathbb{C}$-linear $\sigma_j^{(\ell+1)}$-derivation such that for $1 \leq i <j \leq 6$, $\sigma_j^{(\ell+1)}(X_i^{(\ell+1)})=\lambda_{j, i}X_i^{(\ell+1)}$ and $\delta_j^{(\ell+1)}(X_i^{(\ell+1)})=P_{j, i}^{(\ell+1)}$ with $P_{j, i}^{(\ell+1)}$ defined by
\begin{equation*}
P_{j, i}^{(\ell+1)}=\begin{cases}
0, & j \geq \ell+1;\\
\sum_{\overline{k}=(k_{i+1},..., k_{j-1})}c_{\overline{k}}(X_{i+1}^{(\ell+1)})^{k_{i+1}}\cdot \cdot \cdot(X_{j-1}^{(\ell+1)})^{k_{j-1}}, & j< \ell+1,
\end{cases}
\end{equation*}
where $\lambda_{j,i}$ and $c_{\overline{k}}$ are the same as in the formulas (\ref{lambda}) and (\ref{P}), respectively.

The relations between $X_i^{(\ell)}$ and $X_j^{(\ell+1)}, i,j\in [\![i, 6]\!], \ell \in [\![2, 7]\!]$ play a
very important role in determining the derivations of $U^+$. By making complicated calculation,
we may obtain the next result from (\ref{tangz03}).

\begin{lemma}\label{tzlemma1}
Set $Y_i:=X_i^{(6)}, Z_i:=X_i^{(5)}, U_i:=X_i^{(4)}, T_i:=X_i^{(3)}$ for all $i\in [\![1, 6]\!]$. Then we have
the following equations.

{\rm(1)} $Y_1= X_1-(1-r^{-3}s^{3})^{-1}X_5X_6^{-1},$\

 \ \ \ \ \! $Y_2= X_2-(1-r^{-3}s^3)^{-1}((r^2s)(r^3-s^3)X_5X_3+r\eta X_4)X_6^{-1}$

 \ \ \ \ \ \ \ \ \ $+\frac{(1-r^{-3}s^3)^{-2}}{1+r^{-3}s^3}s^4(r^2-s^2)((r^3-s^3)-r\eta s^{-1}(r-s))X_5^3X_6^{-2},$\

\ \ \ \ \! $Y_3 =X_3-(1-r^{-3}s^3)^{-1}r(r^2-s^2)X_5^2X_6^{-1}$,\

\ \ \ \ \! $Y_4 =X_4-(1-r^{-3}s^3)^{-1}r^3(r-s)(r^2-s^2)X_5^3X_6^{-1}$,

\ \ \ \ \! $ Y_5= X_5,\  Y_6 =X_6 $;

{\rm(2)} $Z_1= Y_1-(1-r^{-1}s)^{-1}Y_3Y_5^{-1}+\frac{(1-r^{-1}s)^{-2}}{(1+r^{-1}s)}r^{-1}sY_4Y_5^{-2},$\

 \ \ \ \ \! $Z_2= Y_2-(1-r^{-1}s)^{-1}r\zeta Y_3^2Y_5^{-1}+\frac{(1-r^{-1}s)^{-2}}{(1+r^{-1}s)}(sY_3Y_4+r^2s^2Y_4Y_3)Y_5^{-2},$\

\ \ \ \ \! $Z_3 =Y_3-(1-r^{-1}s)^{-1}Y_4Y_5^{-1}, \ Z_4=Y_4,\ Z_5=Y_5,\ Z_6=Y_6$;

{\rm(3)} $U_1= Z_1-(1-r^{-3}s^{3})^{-1}r \zeta Z_3^2Z_4^{-1},$\

 \ \ \ \ \! $U_2= Z_2-(1-r^{-3}s^{3})^{-1}r^3\zeta(r-s)Z_3^3Z_4^{-1}$,\

\ \ \ \ \! $U_3=Z_4, \ U_4=Z_4,\ U_5=Z_5,\ U_6=Z_6$;

{\rm(4)} $T_1= U_1-(1-r^{-1}s)^{-1}U_2U_3^{-1},$\

\ \ \ \ \! $T_2=U_2, \ T_3=U_3, \ T_4=U_4,\ T_5=U_5,\ T_6=U_6$;

{\rm(5)} $X_i^{(2)}= T_i$ for all $i\in [\![1, 6]\!]$.
\end{lemma}

Now we will give some useful corollary and  proposition:

\begin{corollary}\label{3.5}
With the notion as above, the following identities hold:

{\rm(1)}\, $X_1=Y_1+(1-r^{-3}s^{3})^{-1}Y_5Y_6^{-1}$,

\quad \ \ $X_2=Y_2+(1-r^{-3}s^{3})^{-1}(r^2s)(r^3-s^3)Y_5Y_3Y_6^{-1}$\

\ \ \ \ \ \ \qquad$+(r^3s)(r^3-s^3)(r^2-s^2)(1-r^{-3}s^{3})^{-2}Y_5^3Y_6^{-2}$\

\ \ \ \ \ \ \qquad$+(1-r^{-3}s^{3})^{-1}r\eta Y_4Y_6^{-1}+(1-r^{-3}s^{3})^{-2}r^4\eta (r-s)(r^2-s^2)Y_5^3Y_6^{-2}$\

\ \ \ \ \ \ \qquad $-\frac{(1-r^{-3}s^{3})^{-2}}{1+r^{-3}s^{3}}s^4(r^2-s^2)((r^3-s^3)-r\eta s^{-1}(r-s))Y_5^3Y_6^{-2}$,

\quad \ \ $X_3=Y_3+(1-r^{-3}s^{3})^{-1}r(r^2-s^2)Y_5^2Y_6^{-1}$,

\quad \ \ $X_4=Y_4+(1-r^{-3}s^{3})^{-1}r^3(r-s)(r^2-s^2)Y_5^3Y_6^{-1}$,

\quad \ \ $X_5=Y_5,X_6=Y_6.$

{\rm (2)} $Y_1=Z_1+(1-r^{-1}s)^{-1}Z_3Z_5^{-1}+(1-r^{-1}s)^{-2}Z_4Z_5^{-2}-\frac{(1-r^{-1}s)^{-2}}{(1+r^{-1}s)}r^{-1}sZ_4Z_5^{-2},$

\quad \ \ \ \,$Y_2=Z_2+(1-r^{-1}s)^{-1}r\zeta Z_3^2Z_5^{-1}+2(1-r^{-1}s)^{-2}r\zeta Z_3 Z_4Z_5^{-2}$

\ \ \ \ \ \ \ \qquad$+(1-r^{-1}s)^{-3}r\zeta Z_4Z_5^{-1}Z_4 Z_5^{-2}-\frac{(1-r^{-1}s)^{-2}}{(1+r^{-1}s)}sZ_3Z_4Z_5^{-2}$

\ \ \ \ \ \ \ \qquad$-\frac{(1-r^{-1}s)^{-3}}{(1+r^{-1}s)}sZ_4Z_5^{-1}Z_4Z_5^{-2}-\frac{(1-r^{-1}s)^{-2}}{(1+r^{-1}s)}r^2s^2Z_4Z_3Z_5^{-2}$

\ \ \ \ \ \ \ \qquad$-\frac{(1-r^{-1}s)^{-3}}{(1+r^{-1}s)}(1-r^{-1}s)^{-1}Z_4^{2}Z_5^{-3};$

\ \ \ \ \ \ $Y_3=Z_3+(1-r^{-1}s)^{-1}Z_4Z_5^{-1}, $

\ \ \ \ \ \ \,$Y_4=Z_4,$  $Y_5=Z_5, $ $Y_6=Z_6.$

{\rm (3)} $Z_1=U_1+(1-r^{-3}s^{3})^{-1}r \zeta U_3^2U_4^{-1}, $

\quad \ \ \ \ \!\,\!$Z_2=U_2+(1-r^{-3}s^{3})^{-1}r^3\zeta (r-s) U_3^3U_4^{-1}, $

\quad \ \ \ \ \!$Z_3=U_3, Z_4=U_4,$  $Z_5=U_5, Z_6=U_6, $

{\rm (4)}\, $U_1=T_1+(1-r^{-1}s)^{-1}T_2T_3^{-1}, $

\quad \ \ \ \ $U_2=T_2, U_3=T_3,$ $U_4=T_4, U_5=T_5, $ $U_6=T_6.$
\end{corollary}

Furthermore, we can easily prove by Lemmas \ref{12}, \ref{tzlemma1} the following proposition, which
describes the relations between the variables $T_1, \cdots, T_6$.
\begin{lemma}\label{relation}
The following identities hold:
\begin{eqnarray*}
&&T_1T_6=s^3T_6T_1, \  T_2T_6=(rs^2)^3T_6T_2, \ T_3T_6=(rs)^3T_6T_3, \\
&&T_4T_6=(r^2s)^3T_6T_4, \ T_5T_6=r^3T_6T_5, \  T_1T_5=rs^2T_5T_1,\\
&&T_2T_5=(rs)^3T_5T_2, \ T_3T_5=r^2sT_5T_3, \ T_4T_5=r^3T_5T_4,\\
&&T_1T_4=(rs)^3T_4T_1, \ T_2T_4=(r^2s)^3T_4T_2,  \ T_3T_4=r^3T_4T_3,\\
&&T_1T_3=r^2sT_3T_1, \ T_2T_3=r^3T_3T_2, \ T_1T_2=r^3T_2T_1.
\end{eqnarray*}
\end{lemma}

For $\ell\in [\![2, 7]\!]$, we denote by $G^{(\ell)}$ the subalgebra of $Frac(U^+)$ generated  by $X_i^{(\ell)}$ with $i \in [\![1, 6]\!]$. Clearly, $G^{(7)}=U^+$. Analogous to Lemma 4.2.3 in \cite{meriaux2010}, we have

(i) $G^{(6)}=\mathbb{C}[X_1^{(6)}][X_2^{(6)};\sigma_2^{(6)}, \delta_2^{(6)}]...[X_5^{(6)};\sigma_5^{(6)}, \delta_5^{(6)}][X_6^{(6)}; \sigma_6^{(6)}]$;

(ii) $G^{(5)}=\mathbb{C}[X_1^{(5)}][X_2^{(5)};\sigma_2^{(5)}, \delta_2^{(5)}]...[X_4^{(6)};\sigma_4^{(6)}, \delta_4^{(6)}][X_5^{(5)};\sigma_5^{(5)}][X_6^{(6)}; \sigma_6^{(6)}]$;

(iii) $G^{(4)}=\mathbb{C}[X_1^{(4)}][X_2^{(4)};\sigma_2^{(4)}, \delta_2^{(4)}][X_3^{(4)};\sigma_3^{(4)}, \delta_3^{(4)}][X_4^{(4)};\sigma_4^{(4)}][X_5^{(5)};\sigma_5^{(5)}][X_6^{(6)}; \sigma_6^{(6)}]$;

(iv) $G^{(3)}=G^{(2)}=\mathbb{C}[T_1][T_2; \sigma_2^{(3)}][T_3 ;\sigma_3^{(3)}][T_4; \sigma_4^{(3)}][T_5; \sigma_5^{(3)}][T_6; \sigma_6^{(3)}]$.

On the other hand, for $\ell\in [\![2, 7]\!]$, we denote by $S_{\ell}$ the multiplicative system generated by $T_{i}$ with $i \in [\![\ell, 6]\!]$. Since $T_i=X_i^{(\ell)}$ for all $i\in [\![\ell, 6]\!]$, we obtain the localization:
$$
G^{\ell}:=G^{(\ell)}S_{\ell}^{-1},
$$
Note that $G^7=G^{(7)}=U^+$.

For $F=(F_1,F_2,\cdots, F_6)$ and $\gamma=(\gamma_1, \cdots, \gamma_6)\in \mathbb{Z}^6$,
we write $F^{\gamma}:=F_1^{\gamma_1}F_2^{\gamma_2}\cdots F_6^{\gamma_6}$.
It is easy to see that the set
$$\left\{(X^{(\ell)})^\gamma|\gamma\in \mathbb{N}^{\ell-1}\times \mathbb{Z}^{7-\ell}\right\}$$
 is a PBW basis of $G^{\ell}$.  Write $\Sigma_{\ell}:=\{T_{\ell}^k|k \in \mathbb{N}\}$. Then by Theorem 3.2.1 in \cite{cauchon}, we have
 $$
 G^{(\ell)}\Sigma^{-1}_{\ell}=G^{(\ell+1)}\Sigma^{-1}_{\ell}.
 $$

In addition, we denote by $G^1$ the subalgebra of $Frac(U^+)$ generated by the variables $T_1^{\pm 1}, T_2^{\pm 1}, T_3^{\pm 1}$, $T_4^{\pm 1},T_5^{\pm 1}, T_6^{\pm 1}$, which clearly form  a quantum torus. Hence we obtain the relation among $G^{\ell}$.
\begin{lemma}\label{3.3}
For all $\ell \in [\![1, 6]\!]$, we have $G^{\ell}=G^{\ell+1}\Sigma_{\ell}^{-1}$.
\end{lemma}
\begin{proof}
By Lemma \ref{tzlemma1}, $T_6=U_6=Z_6=Y_6$. Hence by  Lemma \ref{3.5} we see that $G^{7}\Sigma_{6}^{-1}$ is generated by
$$Y_1, Y_2, Y_3, Y_4, Y_5, Y_6, Y_6^{-1}.$$
Obviously, every elements in $G^{6}$ can be expressed by $G^{7}\Sigma_{6}^{-1}$. This implies  $G^6\subseteq G^{7}\Sigma_{6}^{-1}$. Similarly, we can get $G^\ell \subseteq G^{\ell+1}\Sigma_{\ell}^{-1}$ for every $\ell \in [\![1, 5]\!]$.

Conversely, by Lemma \ref{3.5} we can see that $G^\ell \supseteq G^{\ell+1}\Sigma_{\ell}^{-1}$ for all $\ell \in [\![1, 6]\!]$. The proof is completed.
\end{proof}
As of now,  we can say that $U^+$ may be embedded in $G^1$. From Lemma \ref{3.3}, we have the following corollary.
\begin{corollary}
The following tower of algebras holds:
\begin{eqnarray}
G^{7}=U^+\subset G^{6}=G^7\Sigma_6^{-1}\subset G^{5}=G^6\Sigma_5^{-1}\subset G^{4}=G^5\Sigma_4^{-1}\nonumber\\
\subset G^{3}=G^4\Sigma_3^{-1}\subset G^{2}=G^3\Sigma_2^{-1}\subset G^{1}=G^2\Sigma_1^{-1}.
\end{eqnarray}
\end{corollary}

We give a formula which will be useful for calculating derivations later.

\begin{corollary}\label{3.6}
For every positive integer $k$, one has
\begin{equation}\label{zhongt666}
U_3^{-k}U_1=(r^2s)^{k}U_1U_3^{-k}-d_kU_2U_3^{-(k+1)}
\end{equation}
where  $d_1=(1-r^{-1}s)^{-1}(r^2s-r^3)$ and $d_k=(r^2s)^{k-1}d_1+r^3d_{k-1}$ for $k\ge 2$.
\end{corollary}
\begin{proof}
Lemma \ref{tzlemma1} tells us $T_1= U_1-(1-r^{-1}s)^{-1}U_2U_3^{-1}$ (or equivalently, $U_1=T_1+(1-r^{-1}s)^{-1}U_2U_3^{-1}$) and $T_2=U_2, \ T_3=U_3, \ T_4=U_4,\ T_5=U_5,\ T_6=U_6$.
We shall show the conclusion by using induction on $k$.  By applying Lemma \ref{relation} with the above relations on $T_i$ and $U_i$ repeatedly, we have
\begin{eqnarray*}
U_3^{-1}U_1&=&U_3^{-1}T_1+(1-r^{-1}s)^{-1}U_3^{-1}U_2U_3^{-1}\\
&=&r^2sT_1T_3^{-1}+(1-r^{-1}s)^{-1}T_3^{-1}T_2T_3^{-1}\\
&=&r^2s(U_1-(1-r^{-1}s)^{-1}T_2T_3^{-1})T_3^{-1}+(1-r^{-1}s)^{-1}r^3T_2T_3^{-2}\\
&=&r^2sU_1U_3^{-1}-(1-r^{-1}s)^{-1}(r^2s-r^3)U_2U_3^{-2},
\end{eqnarray*}
which implies (\ref{zhongt666}) is true for $k=1$.  We assume that it is true for
$k=i (i\ge 1)$, i.e., $U_3^{-i}U_1=(r^2s)^{i}U_1U_3^{-i}-d_iU_2U_3^{-(i+1)}$. For $k=i+1$, it follows from Lemmas \ref{tzlemma1},\ref{relation} and the inductive hypothesis that
\begin{eqnarray*}
&&U_3^{-(i+1)}U_1=U_3^{-1}((r^2s)^{i}U_1U_3^{-i}-d_iU_2U_3^{-(i+1)})\\
&&\quad=(r^2s)^{i}U_3^{-1}U_1U_3^{-i}-d_ir^3U_2U_3^{-(i+2)})\\
&&\quad=(r^2s)^{i+1}U_1U_3^{-(i+1)}-((r^2s)^i(1-r^{-1}s)^{-1}(r^2s-r^3)+r^3d_i)U_2U_3^{-(i+2)}\\
&&\quad=(r^2s)^{i+1}U_1U_3^{-(i+1)}-d_{i+1}U_2U_3^{-(i+2)},
\end{eqnarray*}
we obtain that (\ref{zhongt666}) is true for $k=i+1$, and hence the proof is completed.
\end{proof}

By Corollary \ref{3.6}, it is easy to get the following result.

\begin{proposition}\label{3.7}
Denote by $G$ the subalgebra of $G^4$ generated by $U_j$ with $j \neq 3$, $U_4^{-1}$, $U_5^{-1}$ and $U_6^{-1}$, i.e.,
$$G=\langle U_1,U_2,U_4^{\pm 1}, U_5^{\pm 1}, U_6^{\pm 1}\rangle.$$
Then $G^4$ is a free left $G$-module with basis $(U_3^c)_{c \in \mathbb{Z}}$.
\end{proposition}

\subsection{Centers of the algebra $G^{\ell}$}

For
$\gamma=(\gamma_1, \cdots, \gamma_6)$, the following equations will be used many times:
\begin{eqnarray}
3\gamma_2+2\gamma_3+3\gamma_4+\gamma_5=0, \label{zt001}\\
\gamma_3+3\gamma_4+2\gamma_5+3\gamma_6=0, \label{zt002}\\
-3\gamma_1+3\gamma_3+6\gamma_4+3\gamma_5+3\gamma_6=0, \label{zt003}\\
3\gamma_4+3\gamma_5+6\gamma_6=0, \label{zt004}\\
-2\gamma_1+3\gamma_2+3\gamma_4+2\gamma_5+3\gamma_6=0, \label{zt005}\\
-\gamma_1+\gamma_5+3\gamma_6=0, \label{zt006}\\
-3\gamma_1-6\gamma_2-3\gamma_3+3\gamma_5+6\gamma_6=0, \label{zt007}\\
-3\gamma_1-3\gamma_2+3\gamma_6=0,  \label{zt008}\\
-\gamma_1-3\gamma_2-2\gamma_3-3\gamma_4+3\gamma_6=0,  \label{zt009}\\
-2\gamma_1-3\gamma_2-\gamma_3=0,  \label{zt0010}\\
-3\gamma_2-3\gamma_3-6\gamma_4-3\gamma_5=0,  \label{zt0011}\\
-3\gamma_1-6\gamma_2-3\gamma_3-3\gamma_4=0.  \label{zt0012}
\end{eqnarray}

\quad In ring theory, the center $Z(R)$ of a ring $R$ is the subring consisting of the elements $x$ such that $xy = yx$ for all elements $y$ in $R$. Further, we have
\begin{lemma}\label{center}
For every $i\in [\![1, 7]\!]$, one has $Z(G^i)=\mathbb{C}$.
\end{lemma}

\begin{proof}
Recall that, for $\ell\in [\![2, 6]\!]$, the set
$\left\{(X^{(\ell)})^\gamma|\gamma\in \mathbb{N}^{\ell-1}\times \mathbb{Z}^{7-\ell}\right\}$
 is a PBW basis of $G^{\ell}$ in view of Section \ref{section2.3} Note that Lemma 2.7 tells us that
 $Y_i:=X_i^{(6)}, Z_i:=X_i^{(5)}, U_i:=X_i^{(4)}, T_i:=X_i^{(3)}$ and $X_i^{(2)}=T_i$ for all $i\in [\![1, 6]\!]$.

We first prove that $Z(G^1)=Z(G^2)=Z(G^3)=\mathbb{C}$.

If $T^{\gamma}=T_1^{\gamma_1}T_2^{\gamma_2}\cdots T_6^{\gamma_6}$ in the center of the algebra $G^1$ then $T^{\gamma}T_i=T_iT^{\gamma}$ for all $i\in [\![1, 6]\!]$. By Lemma \ref{relation} we have
\begin{eqnarray*}
T_1T^{\gamma}&=&T_1T_1^{\gamma_1}T_2^{\gamma_2}T_3^{\gamma_3}T_4^{\gamma_4}T_5^{\gamma_5}T_6^{\gamma_6}\\
             &=&r^{3\gamma_2+2\gamma_3+3\gamma_4+\gamma_5}s^{\gamma_3+3\gamma_4+2\gamma_5+3\gamma_6}T_1^{\gamma_1}T_2^{\gamma_2}T_3^{\gamma_3}T_4^{\gamma_4}T_5^{\gamma_5}T_6^{\gamma_6}T_1\\
             &=&r^{3\gamma_2+2\gamma_3+3\gamma_4+\gamma_5}s^{\gamma_3+3\gamma_4+2\gamma_5+3\gamma_6}T^{\gamma}T_1.
\end{eqnarray*}
This together with $T_1T^{\gamma}=T^{\gamma}T_1$ gives $r^{3\gamma_2+2\gamma_3+3\gamma_4+\gamma_5}s^{\gamma_3+3\gamma_4+2\gamma_5+3\gamma_6}=1$. Therefore, we obtain Equations (\ref{zt001}) and (\ref{zt002}). Similarly, by
$T_iT^{\gamma}=T^{\gamma}T_i, i=2,\cdots, 6$, we have Equations (\ref{zt003})-(\ref{zt0012}) hold. It follows that $\gamma_1=\gamma_2=\gamma_3=\gamma_4=\gamma_5=\gamma_6=0$. Then the center of $G^{1}$ is equal to $\mathbb{C}$.
On the other hand, note that $G^2$ is generated by $(X_1^{(2)})^{\gamma_1}(X_2^{(2)})^{\gamma_2}\cdots (X_6^{(2)})^{\gamma_6}=T_1^{\gamma_1}T_2^{\gamma_2}\cdots T_6^{\gamma_6}=T^{\gamma}$ where $\gamma_1\in \mathbb{N}$ and $\gamma_i\in \mathbb{Z}$, $i\in [\![2,6]\!]$. We assume that $T^{\gamma} \in Z(G^2)$. So we have $T^{\gamma}T_i=T_iT^{\gamma}$, for $i \in [\![1,6]\!]$.  Similar to the case of $Z(G^1)$, we have Equations (\ref{zt001})-(\ref{zt0012}) hold and so that $\gamma=0$, i.e., $Z(G^2)=\mathbb{C}$. Note that $T_i:=X_i^{(3)}, i\in [\![1, 6]\!]$ yields $G^3$ is generated by $T^\gamma$ with $\gamma_1,\gamma_2\in \mathbb{N}$ and $\gamma_i\in \mathbb{Z}$, $i\in [\![3,6]\!]$. Similar to the case of $Z(G^1)$ or $Z(G^2)$, we also have  $Z(G^3)=\mathbb{C}$.

Next, we prove that $Z(G^4)=Z(G^5)=Z(G^6)=Z(G^7)=\mathbb{C}$.

 Note that $G^4$ is generated by $(X_1^{(4)})^{\gamma_1}(X_2^{(4)})^{\gamma_2}\cdots (X_6^{(4)})^{\gamma_6}=U_1^{\gamma_1}U_2^{\gamma_2}\cdots U_6^{\gamma_6}$ where $\gamma_1,\gamma_2,\gamma_3\in \mathbb{N}$ and $\gamma_4,\gamma_5,\gamma_6\in \mathbb{Z}$. By Lemma \ref{tzlemma1}, $U_1=T_1+(1-r^{-1}s)^{-1}U_2U_3^{-1}$ and $T_2=U_2, \ T_3=U_3, \ T_4=U_4,\ T_5=U_5,\ T_6=U_6$.
 We assume that $a \in Z(G^4)$. Then we have $a U_i=U_ia, i\in [\![1,3]\!]$  and $a T_j^{\pm 1}=T_j^{\pm 1} a$ for $j \in [\![4, 6]\!]$.  Therefore one has $aT_i=T_ia$  for $i \in [\![2, 6]\!]$. It follows by $aU_1=U_1a$ and $aU_1=a(T_1+(1-r^{-1}s)^{-1}U_2U_3^{-1})=aT_1+(1-r^{-1}s)^{-1}aT_2T_3^{-1}=aT_1+(1-r^{-1}s)^{-1}T_2T_3^{-1}a$
that $aT_1=T_1a$. Then we have $a \in Z(G^3)=\mathbb{C}$. Therefore $Z(G^4)=\mathbb{C}$. Similarly, by Lemma  \ref{tzlemma1} we obtain that $Z(G^{5})\subseteq Z(G^{4})=\mathbb{C}$, $Z(G^{6})\subseteq Z(G^{5})=\mathbb{C}$ and $Z(G^{7})\subseteq Z(G^{6})=\mathbb{C}$.

The proof is completed.
\end{proof}

\section{Derivations and the Hochschild Cohomology Group of Degree 1 of $U_{r,s}^+(G_2)$}

\subsection{Derivations of $U_{r,s}^+(G_2)$}
In this section, we denote by $Der(U^+)$ the set of $\mathbb{C}$-derivations of $U^+$.
We now define two derivations $D_5$ and $D_6$ of $U^+$ as follows.
\begin{definition}\label{XGL}
Let $D_5$ and $D_6$ be two derivations of $U^+$ determined by
\begin{eqnarray*}
D_5(X_1)=X_1, & D_5(X_2)=3X_2, & D_5(X_3)=2X_3,\\
D_5(X_4)=3X_4, &  D_5(X_5)=X_5, & D_5(X_6)=0;\\
D_6(X_1)=-X_1, & D_6(X_2)=-2X_2, &D_6(X_3)=-X_3,\\
D_6(X_4)=-X_4, &  D_6(X_5)=0, &D_6(X_6)=X_6.
\end{eqnarray*}
\end{definition}

Our main result is following.
\begin{theorem}\label{maintheo}
Every derivation $D$ of $U^+$ can be uniquely written as
\begin{equation*}
D=ad_g+\mu_5D_5+\mu_6D_6,
\end{equation*}
where $g \in U^+$, $\mu_5, \mu_6 \in \mathbb{C}$ and $D_5, D_6$ are given by Definition \ref{XGL}.
\end{theorem}

The proof of Theorem \ref{maintheo} will be completed by several lemmas as follows. Now we assume that $D \in Der(U^+)$ be a derivation of $U^+$. Recall that $U^+=G^{7}$. Due to Lemma \ref{3.3} and Corollary \ref{3.3}, $D$ extends uniquely to a derivation of $G^6$, $G^5$, $G^4$, $G^3$, $G^2$, $G^1$, respectively. We will still denote the extended derivations by $D$.
 Recall that $G^1$ is generated by the monomials $T^\gamma=T_1^{\gamma_1}T_2^{\gamma_2}\cdots T_6^{\gamma_6}, \gamma\in \mathbb{Z}^6$. Thanks to the result in \cite{Osborn1995}, one knows that the derivation $D$ of the quantum torus $G^1$ can be written as the sum of an inner derivation and a central derivation\cite{Osborn1995}. Namely, there are some element $g\in G^1$ such that
 \begin{equation}\label{zht52}
 D={\rm ad}_g+\theta,
 \end{equation}
 where $\theta$ is a central derivation of $G^1$ which is defined by $\theta (T_i)=\mu_i T_i$ with $\mu_i\in Z(G^1)=\mathbb{C}, i=1, \cdots, 6$.
Although $g\in G^1$, but we will prove that $g \in G^2$, $g \in G^3, \cdots, g\in G^7$ step by step. In addition we also will give the relations among $\mu_1, \cdots, \mu_6$. If $g\in \mathbb{C}$ then
 ${\rm ad} g=0$ for which gives a trivial case. Hence we assume that $g\notin \mathbb{C}$ below. We will prove the following lemmas.

\begin{lemma}\label{lm1}
For all $\ell \in [\![1,3]\!]$, we have $g \in G^{\ell}$.
\end{lemma}
\begin{proof}
The case $\ell=1$ is trivial. We first prove that $g\in G^2$.
For any $z \in \mathbb{C}=Z(G^1)$, it easy to see that $ad_g=ad_{g+z}$. This allows us write $g\in G^1$ as
\begin{eqnarray*}
g=\sum_{\gamma \in \xi}c_{\gamma}T^\gamma=\sum_{\gamma \in \xi}c_{\gamma}T_1^{\gamma_1}T_2^{\gamma_2}\cdots T_6^{\gamma_6}
\end{eqnarray*}
where $c_{\gamma} \in \mathbb{C}$ and $\xi$ is a finite subset of $\mathbb{Z}^6$ with $0\not\in \xi$.
We below will prove that $\gamma_1 \geq 0$ for all $\gamma=(\gamma_1, \cdots, \gamma_6)\in \xi$, and so that $g$ will belong to $G^2$. Set $g=g_++g_-$ where
\begin{equation*}
g_+=\sum_{\gamma \in \xi, \gamma_1 \geq 0}c_{\gamma}T^{\gamma}\in G^2
\end{equation*}
and
\begin{equation*}\label{x_-1}
g_-=\sum_{\gamma \in \xi, \gamma_1 < 0}c_{\gamma}T^{\gamma}.
\end{equation*}
we shall prove that $g_-=0$.
Since $G^2$ is generated by $T_1$, $T_2^{\pm 1}$, $T_3^{\pm 1}$, $T_4^{\pm 1}$, $T_5^{\pm 1}$, $T_6^{\pm 1}$ and $D$ extends uniquely to a derivation of $G^2$, then for any $j \in [\![2,6]\!]$ we must have
\begin{eqnarray*}
D(T_j)&=&ad_g(T_j)+\theta (T_j)\\
&=&g_+T_j-T_jg_{+}+g_-T_j-T_jg_-+\mu_jT_j \in G^2.
\end{eqnarray*}
 Clearly, by $g_+, T_j\in G^2$ we have
\begin{equation}\label{u1}
g_-T_j-T_jg_- \in G^2.
\end{equation}
Using the commutation relation among $T_k$, (\ref{u1}) becomes
\begin{equation}\label{u2}
\sum_{\gamma \in \xi, \gamma_1 < 0}c_{j, \gamma}'c_{\gamma}T^{\gamma+\epsilon_j} \in G^2
\end{equation}
where $\epsilon_j$ denotes the $j$-th element of the canonical basis of $\mathbb{Z}^6$ and $c_{j, \gamma}' \in \mathbb{C}$.
Recall that $T^{\gamma}$ with $\gamma \in \mathbb{N}^1\times \mathbb{Z}^5$ forms a PBW basis of $G^2$. Since $0\not\in \xi$, we must have $c_{j, \gamma}'c_{\gamma}=0$ and thus we have
$$g_-T_j-T_jg_-=\sum_{\gamma \in \xi, \gamma_1 < 0}c_{j, \gamma}'c_{\gamma}T^{\gamma+\epsilon_j}=0$$
 for all $j \neq 1$. In other words, $g_-$ commutes with those $T_j$ for all $j \neq 1$.   This, together with (\ref{x_-1}), gives that $c_\gamma T_j T^\gamma=c_\gamma T^\gamma T_j$ for every $j\in [\![2,6]\!]$, where $\gamma$ is given by the sum $g_-=\sum_{\gamma \in \xi, \gamma_1 < 0}c_{\gamma}T^{\gamma}$. If $c_\gamma\neq 0$, then similar to the proof of Lemma \ref{center} we obtain that Equations (\ref{zt003})-(\ref{zt0012}) hold which implies $\gamma=0$. This impossible since $0\notin \xi$. Therefore, every $c_\gamma$ is equal to zero. In other words, $g_{-}=0$, that is $g=g_+\in G^2$.

Next, we prove that $g\in G^{3}$. As we already see that  $g\in G^2$, so we can write $g$ as $g_=\sum_{\gamma \in \xi' }c_{\gamma}T^{\gamma}$, where $\xi'\subseteq \mathbb{N}^1 \times \mathbb{Z}^{5}$. Set
\begin{equation*}
g_+=\sum_{\gamma \in \xi', \gamma_2 \geq 0}c_{\gamma}T^{\gamma}, \ \ g_-=\sum_{\gamma \in \xi', \gamma_2 < 0}c_{\gamma}T^{\gamma}.
\end{equation*}
Similarly, we can get $g_-$ commutes with $T_j$ such that $j \neq 2$. Note that the system of Equations (\ref{zt001}), (\ref{zt002}) and (\ref{zt005})-(\ref{zt0012}) for $\gamma_1, \cdots, \gamma_6$ has only zero solutions. In the same way with the above proof process, we also get $g_-=0$ which yields $g=g_+\in G^3$.
\end{proof}

\begin{lemma}\label{lm2}
(i)\ $\mu_2-\mu_3-\mu_1=0$. \ \ (ii)\ $g \in G^4$.\ \

(iii)\ $D(U_i)=ad_g(U_i)+\mu_i U_i$ for all $i \in [\![1, 6]\!]$.
\end{lemma}
\begin{proof}
(i)\ By Lemmas \ref{lm1} and \ref{relation}, we can set $g_-=\sum_{\gamma \in \xi''}c_{\gamma}U^{\gamma}$, where $\xi'' \subseteq \mathbb{N}^{2} \times \mathbb{Z}^4$ with $0 \not \in \xi''$. Set $g_+=\sum_{\gamma \in \xi'', \gamma_3 \geq 0}c_{\gamma}U^{\gamma}$ and $g_-=\sum_{\gamma \in \xi'', \gamma_3<0}c_{\gamma}U^{\gamma}$. From proposition \ref{3.7}, $g_-$ can be written as
\begin{equation}
g_-=\sum_{c=c_0}^{-1}b_cU^{c}_3
\end{equation}
where $c_0 < 0$ and $b_c \in G$.
 Since $D$ can extend uniquely to a derivation of $G^4$, then $D(U_1)\in G^{4}$. From corollary \ref{3.5},  $U_1=T_1+(1-r^{-1}s)^{-1}T_2T_3^{-1}$, which implies
\begin{eqnarray*}\label{D(U1)}
D(U_1)&=&ad_g(U_1)+ \theta(U_1)\\
&=&g_+U_1-U_1g_++g_-U_1-U_1g_-+\mu_1 T_1+(1-r^{-1}s)^{-1}(\mu_2-\mu_3)T_2T_3^{-1}\in G^4.
\end{eqnarray*}
Furthermore, we have
\begin{eqnarray*}
&&g_-U_1-U_1g_-+\mu_1 T_1+(1-r^{-1}s)^{-1}(\mu_2-\mu_3)T_2T_3^{-1}\\\
&=&\sum_{c=c_0}^{-1}b_cU^c_3U_1-\sum_{c=c_0}^{-1}U_1b_cU^{c}_3+\mu_1U_1+(1-r^{-1}s)^{-1}(\mu_2-\mu_3-\mu_1)T_2T_3^{-1} \in G^4.
\end{eqnarray*}

Since $U^{\gamma}$ with $\gamma \in \mathbb{N}^3 \times \mathbb{Z}^3$ forms a PBW basis of $G^4$, it follows that
\begin{equation*}
\mu_2-\mu_3-\mu_1=0, \ \ \ \ \ \sum_{c=c_0}^{-1}b_cU_3^cU_1-\sum_{c=c_0}^{-1}U_1b_cU_3^c \in G^4.
\end{equation*}

(ii)\ We prove that $g \in G^4$. Suppose that $g_- \neq 0$, then there is some $b_{c_0} \neq 0$. By Corollary \ref{3.6} we have
\begin{eqnarray*}
g_-U_1-U_1g_-&=&\sum_{c=c_0}^{-1}b_cU^{c}_3U_1-U_1\sum_{c=c_0}^{-1}b_cU^{c}_3\\
&=&\sum_{c=c_0}^{-1}(r^{-2c}s^{-c}b_cU_1-U_1b_c)U_3^c-\sum_{c=c_0}^{-1}d_{-c}b_cU_2U_3^{c-1}\in G^4.
\end{eqnarray*}
Hence, we must have $b_{c_0}=0$. This is a contradiction.

(iii)\ The case for $i \neq 1$ is trivial since $U_i=T_i$ and $D(T_i)=ad_g(T_i)+\mu_iT_i$. Due to $U_1=T_1+(1-r^{-1}s)^{-1}T_2T_3^{-1}$ and $\mu_1=\mu_2-\mu_3$, we obtain
\begin{eqnarray*}
D(U_1)&=&ad_g(U_1)+\mu_1T_1+(1-r^2s)^{-1}(\mu_2-\mu_3)T_2T_3^{-1}\\
&=&ad_g(U_1)+\mu_1T_1+(1-r^2s)^{-1}\mu_1T_2T_3^{-1}\\
&=&ad_g(U_1)+\mu_1U_1.
\end{eqnarray*}
This completes the proof.
\end{proof}

Similarly, we can prove the following lemma:
\begin{lemma}\label{lm3}

(i)\ $3\mu_3-\mu_4-\mu_2=0$, $2\mu_3-\mu_4-\mu_1=0$, $\mu_4-\mu_5-\mu_3=0$, $2\mu_3-\mu_5-\mu_2=0$, $\mu_3+\mu_4-2\mu_5-\mu_2=0$, $2\mu_4-3\mu_5-\mu_2=0$, $\mu_3-\mu_5-\mu_1=0$, $\mu_4-2\mu_5-\mu_1=0$, $\mu_5+\mu_3-\mu_6-\mu_2=0$, $\mu_4-\mu_6-\mu_2=0$, $3\mu_5-2\mu_6-\mu_2=0$, $\mu_5-\mu_6-\mu_1=0$, $3\mu_5-\mu_6-\mu_4=0$, $2\mu_5-\mu_6-\mu_3=0$.

(ii)\ $g \in G^{\ell}$ for all $\ell \in [\![5, 7]\!]$.

(iii)\ $D(A_i)=ad_g(A_i)+\mu_iA_i$ for all $ i \in [\![1, 6]\!]$ and $A \in \{Z, Y, X\}$.
\end{lemma}

Now we ready to give the proof of Theorem \ref{maintheo}.

{\bf The proof of Theorem \ref{maintheo}:}
Suppose that $D$ is a derivations of $U^+$.  Based on the discussion at the beginning of this section with Lemmas \ref{lm1}-\ref{lm3}, we know that $D$
can be written as
\begin{eqnarray*}
D=ad_g+\theta,
\end{eqnarray*}
where $g \in U^+$ and $\theta$ is a derivation of $U^+$ determined by $\theta(X_i)=\mu_i X_i, i\in [\![1, 6]\!]$, in which $\mu_1, \cdots, \mu_6\in \mathbb{C}$ satisfying $15$ Equations given by (i) of Lemma \ref{lm2} and (i) of Lemma \ref{lm3}. By a simple computation, we obtain that
$$
\mu_1 = \mu_5-\mu_6,  \mu_2 = 3 \mu_5-2 \mu_6,  \mu_3 = 2 \mu_5-\mu_6, \mu_4 = 3 \mu_5-\mu_6,
$$
and $\mu_5,\mu_6$ are any complex numbers. Now let  $D_5$ and $D_6$ be two derivations of $U_+$ given by Definition \ref{XGL}, then it easy to see that $D$
can be written as
\begin{eqnarray*}
D=ad_g+\mu_5D_5+\mu_6 D_6.
\end{eqnarray*}
  Next, we show that the decompositions is unique. Assume that $ad_g+\mu_5D_5+\mu_6D_6=0$ as a derivation of $U^+$, to finish the proof, we need to show that $\mu_5=\mu_6=0$ and $ad_g=0$.
   Let us set a derivation $\theta:=\mu_5D_5+\mu_6D_6\in {\rm Der}(U^+)$. Then $\theta$ uniquely extends to a derivation $\tilde{\theta}$ of the quantum torus $G^1$ and $ad_g+\widetilde{\theta}=0$.
In addition, we can also get

$\tilde{\theta}(T_1)=(\mu_5-\mu_6)T_1, \qquad \ \, \tilde{\theta}(T_2)=(3\mu_5-2\mu6)T_2, \qquad \ \ \tilde{\theta}(T_3)=(2\mu_5-\mu_6)(T_3),$

$\tilde{\theta}(T_4)=(3\mu_5-\mu_6)T_4, \qquad  \tilde{\theta}(T_5)=(\mu_5)T_5, \qquad\qquad\qquad    \tilde{\theta}(T_6)=(\mu_6)T_6$.

Hence  $\tilde{\theta}$ is a central derivation of $G^1$. Thanks to \cite{Osborn1995}, we can obtain $ad_g=0=\theta$ and thus $\mu_5=\mu_6=0$.  This proves
the uniqueness of the decomposition of the derivation $D$.
The proof of Theorem \ref{maintheo} is completed.

\subsection{The Hochschild cohomology group of degree 1 of $U^+$}

Recall that the Hochschild cohomology group of degree 1 of $U^+$ is denote by $HH^{1}(U^+)$, which defined by
\begin{equation*}
HH^1(U^+):={\rm Der}(U^+)/{\rm Inn Der}(U^+),
\end{equation*}
where ${\rm Inn Der}(U^+):=\{ad_g| g \in U^+\}$ is the Lie algebra of inner derivations of $U^+$. It is easy to see that $HH^1(U^+)$ is a module over $HH^0(U^+):=Z(U^+)=\mathbb{C}$. In particular, we have the following theorem:
\begin{theorem}
The first Hochschild cohomology group $HH^1(U^+)$ of $U^+$ is a two-dimensional vector space spanned by ($\overline{D_5}$, $\overline{D_6}$).
\end{theorem}

\vspace{1cm}
\noindent Yongye Zhong\\
School of Science,\\
Harbin Engineering University,\\
 Harbin 150001, China\\
{\it E-mail address}: 845630692@qq.com

\vspace{4mm}

\noindent Xiaomin Tang\\
School of  Mathematical Science,\\
Heilongjiang University, \\
  Harbin 150080, China\\
\emph{and}\\
School of Science,\\
Harbin Engineering University,\\
 Harbin 150001, China\\
{\it E-mail address}: x.m.tang@163.com


\begin{thebibliography}{20}
\bibitem{Bergeron} Bergeron. N., Gao, Y., Hu N. H., Dringeld's doubles and Lusztig's symmetries of two-parameter quantum groups. \textit{J  Algebra}, $\mathbf{301}$, 378-405(2006).

\bibitem{Benk1} Benkart G., Witherspoon S. ,Two-parameter quantum groups and Drinfel'd doubles. \textit{Algebr Represent Th}, $\mathbf{7}$(3), 261-286(2004).

\bibitem{Benk2} Benkart G., Witherspoon S., Representations of two-parameter quantum groups and Schur-Weyl duality. arXiv preprint math/0108038, 2010.

\bibitem{Benk3} Benkart G., Kang S. J., Lee K. H., On the centre of two-parameter quantum groups. $\textit{P Roy Soc Edinb A}$, $\mathbf{136}$(3), 445-472(2006).

\bibitem{Hu2010}Hu N. H., Wang X. L., Convex PBW-type Lyndon Bases and Restricted Two-Parameter Quantum Group of Type B. \textit{J  Geom Phys.},  $\mathbf{60}$(3), 430-453(2010).

\bibitem{Hu2007} Hu N. H., Shi Q., The two-parameter quantum group of exceptional type $G_2$ and Lusztig symmetries. \textit{Pacific J  math}, $\mathbf{230}$(2), 327-345(2007).

\bibitem{Hu2009}Hu N. H., Wang X. L., Convex PBW-type Lyndon bases and restricted two-parameter quantum groups of type $G_2$. \textit{Pacific J  math}, $\mathbf{241}$(2), 243-273(2009).

\bibitem{Fan}Fan Z. B., Li Y. Q., Two-parameter quantum algebras, Canonical Bases and Categorifications. \textit{Int Math Res Notices}, $\mathbf{2015}$(16), 7016-7062(2014).

\bibitem{tangxin0} Tang X., Ringel-Hall algebras and two-parameter quantized enveloping algebras. \textit{Pacific J Math}, $\mathbf{247}$(1), 213-240(2010).

\bibitem{tangxin1} Tang X., Derivations of two-parameter quantized enveloping algebra $U_{r, s}^{+}(B_2)$,  \textit{Commun Algebra}, $\mathbf{41}$(12), 4602-4621(2013).

\bibitem{tangxin2} Tang X., (Hopf)algebra automorphisms of the hopf algebra $\check{U}^{\geq 0}_{r, s}(sl_3)$, \textit{Commun Algebra},  $\mathbf{41}$(8), 2996-3012(2013).

\bibitem{tangxin3} Tang X., Automorphisms of the two-parameter Hopf algebra $\check{U}^{\geq 0}_{r, s}(G_2)$. arXiv preprint arXiv:1106.1908, 2011.

\bibitem{Osborn1995}Osborn J. M., Passman D. S., Derivations of skew polynomial rings. \textit{J Algebra}, $\mathbf{176}$(2), 417-448(1995).

\bibitem{limin}Li M., Wang X. L., Derivations and automorphism of the positive part of the two-parameter quantum group $U_{r, s}(B_3)$. \textit{Acta math sin}, $\mathbf{33}$(2),   235-251(2017).

\bibitem{Lusztig1990} Lusztig G., Quantum groups at roots of $1$. \textit{Geomtriae Dedicata}, $\mathbf{35}$(1-3), 89-113(1990).

\bibitem{cauchon}Cauchon G., Effacement des $\rm{d\acute{e}rivation}$ et spectres premiers des $alg\grave{e}bres$ quantiques. \textit{J Algebra}, $\mathbf{260}$(2), 476-518(2003).

\bibitem{meriaux2010}M$\rm{\acute{e}}$riaux A., Cauchon diagrams for quantized enveloping algebras. \textit{J Algebra} $\mathbf{323}$(4), 1060-1097(2010).

\bibitem{Christian}Kassel C., Quantum groups [M]. Springer Science and Business Media, 2012.



\end{thebibliography}
\end{document}